\documentclass{amsart} 
\usepackage[usenames,dvipsnames,svgnames,table]{xcolor}
\usepackage{amsmath, amssymb, mathrsfs}

\usepackage[all, cmtip]{xy} 
\usepackage{tikz}
 \usetikzlibrary{cd}
 \tikzset{
  symbol/.style={
    draw=none,
    every to/.append style={
      edge node={node [sloped, allow upside down, auto=false]{$#1$}}}
      }
      }

\usepackage{amsthm}
	\theoremstyle{definition} % bold numbering, plain content
	\newtheorem{definition}{Definition}[section]
	\newtheorem{example}[definition]{Example}
	\theoremstyle{plain} % bold numbering, italic content
	\newtheorem{theorem}[definition]{Theorem}
	\newtheorem{lemma}[definition]{Lemma}
	\newtheorem{proposition}[definition]{Proposition}
	\newtheorem{corollary}[definition]{Corollary}
	
	\theoremstyle{remark} % italic numbering, plain content
	\newtheorem{remark}[definition]{Remark}

\usepackage{enumitem} 

\renewcommand{\AA}{\mathbb A}

\newcommand{\CC}{\mathbb C}

\newcommand{\QQ}{{\mathbb Q}}

\newcommand{\TT}{{\mathbb T}}

\newcommand{\ZZ}{{\mathbb Z}}

\newcommand{\cH}{{\mathcal H}}

\newcommand{\frakM}{\mathfrak M}

\newcommand{\fraka}{\mathfrak a}

\newcommand{\frakp}{\mathfrak p}

\newcommand{\Qbar}{\overline{\QQ}}

\newcommand{\Qpbar}{\Qbar_p}

\newcommand{\cris}{\mathrm{cris}}

\usepackage{colonequals}
\newcommand{\defeq}{\colonequals} %definition %cannot be used in abstract

\newcommand{\isom}{\cong} %isomorphism
\newcommand{\inj}{\hookrightarrow}

\newcommand{\directsum}{\oplus} %direct sum
\newcommand{\tensor}{\otimes} %tensor
\DeclareMathOperator\GL{GL} %general linear group
\newcommand{\ord}{\mathrm{ord}} %ordinary
\DeclareMathOperator\Ind{Ind} %induced representation
\newcommand{\id}{\mathrm{id}} %identity map
\DeclareMathOperator\val{val} %valuation
\DeclareMathOperator\Gal{Gal} %Galois groups
\DeclareMathOperator\Tr{Tr} %trace
\DeclareMathOperator\Nm{Nm} %norm
\DeclareMathOperator\Frob{Frob} %Frobenius
\newcommand{\cyc}{\mathrm{cyc}}
\DeclareMathOperator\loc{loc}

\usepackage[OT2,T1]{fontenc}
\DeclareSymbolFont{cyrletters}{OT2}{wncyr}{m}{n}
\DeclareMathSymbol{\Sha}{\mathalpha}{cyrletters}{"58} %"
\DeclareRobustCommand{\SkipTocEntry}[5]{} % To omit entries in Table of Contents

\usepackage[bookmarks=false, % Acrobat bookmark is created, default=true
        colorlinks, % color the links instead of using boxes, default=false
        linkcolor=red, % color for normal internal links, default=red
        anchorcolor=red, % color for anchor text (text than can be clicked on and link), default=black
        citecolor=red, % color for citing bibliography, default=green
        urlcolor=red] % color for linked url, default=magenta
        {hyperref}  % Should be the last loaded package

\newcounter{counter}
\begin{document}
\title[Fourier coefficients of the generalized eigenform for a CM form]
{Fourier coefficients of the overconvergent generalized eigenform associated to a CM form}
\author{Chi-Yun Hsu}
\address{Department of Mathematics, Harvard University, 1 Oxford Street, Cambridge, MA 02138, USA}
\email{chiyun@math.harvard.edu}
\curraddr{Department of Mathematics, University of California, Los Angeles, Los Angeles, CA 90095, USA}
\email{cyhsu@math.ucla.edu}

%\thanks{}
\date{\today}

%%%%%%%%%%%%%%%%%%%%%%%%%%%%%%%%%%%%%

\begin{abstract}
Let $f$ be a modular form with complex multiplication.
If $f$ has critical slope, then Coleman's classicality theorem implies that there is a $p$-adic overconvergent generalized Hecke eigenform with the same Hecke eigenvalues as $f$. 
We give a formula for the Fourier coefficiets of this generalized Hecke eigenform.
We also investigate the dimension of the generalized Hecke eigenspace of $p$-adic overconvergent forms containing $f$.
\end{abstract}

\maketitle

\tableofcontents

%%%%%%%%%%%%%%%%%%%%%%%%%%%%%%%%%%%%%

%\setcounter{section}{}

\section{Introduction}
Let $f$ be a modular form of weight $k\geq2$.
Assume that $f$ has complex multiplication (CM) and is of critical slope.
Then it is known that $f$ is inside the image of the operator
$\theta^{k-1}$ on $p$-adic overconvergent modular forms, where $\theta = q\frac{d}{dq}$ on $q$-expansions (\cite[Prop.\ 7.1]{Col96}).
One can deduce from Coleman's classicality theorem (\cite[Thm.\ 7.2]{Col96}) that there exists an overconvergent generalized Hecke eigenform $f'$ with the same Hecke eigenvalues as $f$, but $f'$ is not a scalar multiple of $f$.
Since a CM modular form has very explicit description of its Fourier coefficients, it is feasible that it is also possible to explicitly describe the Fourier coefficients of $f'$.
The main theorem of this paper expresses the Fourier coefficients of $f'$ in terms of the value of certain Galois cohomology classes.

Note that in an analogous situation where $f$ is of weight $1$, $p$-regular, and has real multiplication (RM), Cho--Vatsal showed that a $p$-adic overconvergent generalized Hecke eigenform $f'$ exists (\cite{CV}).
In this case, Darmon--Lauder--Rotger explicitly described the Fourier coefficients of $f'$ (\cite{DLR15}) by translating the question into Galois deformation theory, relating the Fourier coefficients to certain Galois cohomology classes, and using class field theory to explicitly describe the relevant Galois cohomology classes.
Our method is the same as \cite{DLR15} in the sense that we also employ Galois deformation theory.
However, unlike the case of weight $1$, the Galois representation associated to a Hecke eigenform of weight $\geq2$ does not have finite image, so we cannot use class field theory to get an even more explicit description of the Galois cohomology classes involved.

In the following we introduce the setup to state our main theorem more precisely.
Let $K$ be an imaginary quadratic field with discriminant $D_K$.
Fix a complex embedding $\sigma\colon K \inj \CC$.
Let $k\geq2$ be an integer and 
\[
 \psi \colon K^\times \backslash \AA_K^\times \rightarrow \CC^\times
\] 
be a Grossencharacter of $K$ of infinity type $\sigma^{k-1}$ with conductor $\frakM$.
Let $f_0$ be the CM modular form associated to $\psi$.
Then $f_0$ is a normalized Hecke eigenform of weight $k$, level $\Gamma_1(N)$ with $N \defeq |D_K| \Nm_{K/\QQ} \frakM$, and nebentypus 
\[
 \chi \colon (\ZZ/N)^\times \rightarrow \Qbar \quad a \rightarrow \chi_K(a) \psi((a)),
\]
where $\chi_K \colon (\ZZ/D_K)^\times \rightarrow \{\pm1\} \subset \CC^\times$ is the (odd) Dirichlet character associated to $K$.
The $q$-expansion of $f_0$ at the $\infty$-cusp is 
\[
 f_0(z)  = \sum_{\fraka} \psi(\fraka) q^{\Nm_{K/\QQ}\fraka},
\]
where $\fraka$ run over all integral ideals of $K$ relatively prime to $\frakM$.

Fix a rational prime $p$ which does not divide $N$ and splits in $K$.
Given an isomorphism $\iota_p \colon \CC \xrightarrow{\sim} \Qpbar$, $\iota_p \circ \sigma \colon K \inj \Qpbar$ determines a prime of $K$ above $p$, which we denote by $\frakp$.
Let $\frakp'$ be another $p$-adic prime of $K$.
Then $\{ \val_p(\psi(\frakp)), \val_p(\psi(\frakp')) \} = \{ 0, k-1 \}$.
We fix some $\iota_p$ such that $\val_p(\psi(\frakp)) = 0$.
Let 
\[
 f(z) = f_0(z) - \psi(\frakp) f_0(pz).
\]
This is a modular form of weight $k$ and level $\Gamma_1(N)\cap \Gamma_0(p)$, and its $U_p$-slope is $k-1$, also referred to as the \emph{critical} slope.
One calls $f$ the \emph{critical $p$-stabilization of $f_0$}. 
We denote its $q$-expansion by 
\[
 f(z) = \sum_{n\geq1} a_n q^n.
\]

Let $\cH^{Np}$ be the abstract Hecke algebra generated over $\QQ_p$ by $T_\ell$ with $\ell\nmid Np$ 
and $U_p$.
Let $I_f$ be the ideal of $\cH^{Np}$ given by the Hecke eigenform $f$, i.e. if we let
\[
 \varphi_f \colon \cH^{Np} \rightarrow \CC_p, \quad 
T_\ell \mapsto a_\ell,
U_p \mapsto \psi(\frakp'), 
\]
then $I_f = \ker \varphi_f$.

Let $M_k^\dagger(\Gamma_1(N),\CC_p)$ (resp.\ $S_k^\dagger(\Gamma_1(N),\CC_p)$) denote the space of (resp.\ cuspidal) overconvergent modular forms of weight $k$ and of tame level $\Gamma_1(N)$ with coefficients in $\CC_p$.
As mentioned before, $S_k^\dagger(\Gamma_1(N),\CC_p)[I^2]$ is $2$-dimensional by \cite[Thm.\ 7.2]{Col96}.
Let $f'\in S_k^\dagger(\Gamma_1(N),\CC_p)[I^2]$ which is not a scalar multiple of $f$.
We write the $q$-expansion of $f'$ as
\[
 f'(z) = \sum_{n\geq1} a_n' q^n.
\]
We further require that $f'$ is \emph{normalized}, in the sense that $a_1'=0$ and $a_{\ell_0}'=1$ where $\ell_0$ is the smallest rational prime inert in $K$; then $f'$ is unique.

We make the assumption that $\psi(\frakp)(\psi^c)^{-1}(\frakp)\neq p$ through out the article.

Our main theorem is a description of $a_\ell'$ in terms of Galois cohomology classes.
\begin{theorem}[Proposition~\ref{prop:ap}, Corollary~\ref{cor:al}] \label{thm:main3}
 Let $f$ and $f'$ be as above.
 Then 
\begin{enumerate}
 \item $a_p' = 0$ 
 \item $a_\ell' = 
\begin{cases} 
 0 & \text{ if $\ell\nmid Np$ and $\ell$ splits in $K$} \\
 \psi(c  \Frob_\ell) \cdot e(\Frob_\ell^2)  & \text{ if $\ell\nmid Np$ and $\ell$ is inert in $K$}.
\end{cases}$ \\
Here $c\in\Gal_\QQ$ is the complex conjugation, and $e\in H^1_{\frakp}(\Gal_K,\psi(\psi^c)^{-1})$ is the unique cohomology class unramified outside $\frakp$ mapping $\Frob_{\ell_0}^2\in \Gal_K$ to $\psi(c \Frob_{\ell_0})^{-1}$, where $\ell_0$ is the smallest rational prime which is inert in $K$.
\end{enumerate}
\end{theorem}

The idea of the proof is as follows.
Let $\rho\colon \Gal_\QQ \rightarrow \GL(V) \isom \GL_2(\CC_p)$ be the $p$-adic Galois representation associated to $f$.
Hence $a_\ell = \Tr(\rho(\Frob_\ell))$ for all $\ell\nmid Np$.
Observe that one may regard $\widetilde{f} \defeq f+f' \varepsilon$ as a Hecke eigenform in $S_k^\dagger(\Gamma_1(N),\CC_p[\varepsilon]/\varepsilon^2)$.
Hence we also have a $p$-adic Galois representation
\[
 \widetilde\rho\colon \Gal_\QQ \rightarrow \GL(\widetilde V) \isom \GL_2(\CC_p[\varepsilon]/\varepsilon^2)
\]
such that 
\[
 a_\ell + a_\ell'\varepsilon = \Tr(\widetilde\rho(\Frob_\ell)), \forall \ell\nmid Np.
\]
Note that $\widetilde{\rho}$ is a first order deformation of $\rho$ with the same Hodge--Tate--Sen weights as $\rho$. 
Moreover, coming from a $p$-adic overconvergent Hecke eigenform, $\widetilde{\rho}$ extends certain crystalline period of $\widetilde{\rho}$.
Hence one can use Galois deformation theory and realize $\widetilde{\rho}$ in the Galois cohomology $H^1(\Gal_\QQ, V\tensor V^\ast)$ with specific local conditions.
We further use the inflation-restriction sequence to reduce the Galois cohomology of $\Gal_\QQ$ to that of $\Gal_K$. 
This makes the Galois cohomology calculation easier because $\left.V\right|_{\Gal_K}$ splits into two characters.

%We mention some related works.
%In the case where the weight $k=1$, Cho--Vatsal showed that a $p$-adic overconvergent generalized Hecke eigenform exists when the modular form has real muliplication (RM) (\cite{CV}), and Darmon--Lauder--Rotger investigated the Fourier coefficients of the generalized Hecke eigenform in this case (\cite{DLR15}).
%Betina generalized this to Hilbert modular forms of parallel weight $1$ with RM (\cite{Bet16}).

We have mentioned that Darmon--Lauder--Rotger investigated the Fourier coefficients of the generalized Hecke eigenform in the case of weight $1$ $p$-regular RM forms (\cite{DLR15}).
Betina generalized this to Hilbert modular forms of parallel weight $1$ with RM and $p$-regular (\cite{Bet16}).
More generally, when one has an explicit description of a Hecke eigenform $f$, and one knows that there is a generalized Hecke eigenform $f'$ associated to it (in other words, the eigenvariety is non-\'etale at the point corresponding to $f$), it might be possible to use the same method to obtain an explicit description of the Fourier coefficients of $f'$. 
For example, one knows that a weight $1$ $p$-irregular CM form gives a non-\'etale point on the eigencurve (\cite{BD}).

Our second result is an investigation of the dimension of the generalized Hecke eigenspace containing $f$, namely $S^\dagger_k(\Gamma_1(N),\CC_p)[I^\infty]$.

\begin{proposition}[Prop.\ \ref{prop:ramdeg}] \label{prop:main}
Let $\bar{\rho}$ be the reduction of $\rho_f$ modulo $p$.
Assume that $\bar\rho$ is absolutely irreducible after restricted to $\Gal_{\QQ(\mu_p)}$.
Then the following are equivalent.
\begin{enumerate}
 \item $\dim S^\dagger_k(\Gamma_1(N),\CC_p)[I^\infty]=2$.
 \item The restriction to decomposition group at $\frakp'$
\[
H^1_{\frakp}(K,\psi(\psi^c)^{-1}) \rightarrow H^1(K_{\frakp'},\psi(\psi^c)^{-1})
\] 
is an isomorphism (equivalently, non-zero).
\end{enumerate}
\end{proposition}
The idea of the proof is to first relate the dimension of the generalized Hecke eigenspace containing $f$, to the dimension of the generalized Hecke eigenspace containing $g$, where $g\in M_{2-k}^\dagger(\Gamma_1(N),\CC_p)$ such that $f=\theta^{k-1}(g)$.
Then one may relate the latter to Galois cohomology, by considering ordinary Galois deformations of $\rho_g$, the Galois representation associated to $g$.

\subsection*{Structure of the paper}
In Section~\ref{sec:fourier}, we prove Theorem~\ref{thm:main3}.
Section~\ref{sec:num} contains numerical examples for the Fourier coefficients of the generalized Hecke eigenform. 
In Section~\ref{sec:ramdeg}, we prove Proposition~\ref{prop:main}.

\section{Fourier coefficients of the generalized Hecke eigenform} \label{sec:fourier}

Let $f=\sum_{n\geq0}a_nq^n$ be as in the introduction, the critical $p$-stabilization of a CM modular form associated to a Grossencharacter $\psi$.

Let $\rho \colon \Gal_\QQ \rightarrow \GL(V)$ be the $p$-adic Galois representation associated to $f$.
Here $V$ is a $2$-dimensional $E$-vector space, where $E$ is the finite extension of $\QQ_p$ generated by the Fourier coefficients $a_n$ of $f$.
We have $\rho \isom \Ind_{\Gal_K}^{\Gal_\QQ} \psi$, where we have abused notation to let $\psi$ also denote the $p$-adic character of $\Gal_K$ associated to the Grossencharacter $\psi$.
In particular, there exists a basis $\{v_1,v_2\}$ of $V$ such that
\[
 \left.\rho\right|_{\Gal_K} \sim 
\begin{pmatrix}\psi&\\&\psi^c\end{pmatrix},
\]
 where $c\in \Gal_\QQ$ is the complex conjugation and $\psi^c(g)\defeq\psi(c g c^{-1})$.
If we choose $v_2$ such that $v_2 = \rho(c) v_1$, then we also have
\[
  \left.\rho\right|_{\Gal_\QQ \setminus \Gal_K} \sim 
\begin{pmatrix}&\eta'\\ \eta&\end{pmatrix},
\]
where $\eta$ and $\eta'$ are functions on $\Gal_\QQ\setminus \Gal_K$ given by $\eta(\cdot) = \psi(c\cdot)$ and $\eta'(\cdot) = \psi(\cdot c^{-1})$.

We have $D^+_\cris(\left.V\right|_{\Gal_{\QQ_p}})^{\varphi=a_p}\neq0$, where $D^+_\cris(\cdot) \defeq (\cdot \tensor_{\QQ_p} B_\cris^+)^{\Gal_{\QQ_p}}$.
Moreover, we have $D^+_\cris(\left.V\right|_{\Gal_{\QQ_p}})^{\varphi=a_p} = D^+_\cris(\left.\psi^c\right|_{\Gal_{K_\frakp}})^{\varphi=a_p}$.
This is because $\left.V\right|_{\Gal_{\QQ_p}} = \left.V\right|_{\Gal_{K_\frakp}} = \psi \directsum \psi^c$, $a_p = \psi(\frakp') = \psi^c(\frakp)$ has $p$-adic valuation $k-1$, and $\left.\psi^c\right|_{\Gal_{K_\frakp}}$ has Hodge--Tate weight $k-1$, 

\subsection{Galois cohomology calculation} \label{subsec:Galoiscohom}
Let $D$ be the deformation functor on the category of Artinian local $E$-algebras with residue field $E$, so that for such $E$-algebra $A$, 
$D(A)$ is the set of strict equivalence classes of continuous representations $\rho_A \colon \Gal_\QQ \rightarrow \GL(V_A)$ deforming $\rho$ such that
\begin{enumerate}
 \item
for all primes $\ell$ not dividing $p$, $\rho_A$ and $\rho$ are the same after restricting to the inertia subgroup at $\ell$,
 \item $\left.V_A\right|_{\Gal_{\QQ_p}}$ has a constant Hodge--Tate--Sen weight $0$, and  
 \item 
there exists a lift $\widetilde{a}_p\in A$ of $a_p$ such that 
$D^+_\cris(\left.V_A\right|_{\Gal_{\QQ_p}})^{\varphi=\widetilde{a}_p}\neq 0$.
\end{enumerate}
Let $D^0$ be the sub-functor of $D$ of deformations with the additional condition
\begin{enumerate}[resume]
 \item 
$\rho_A$ has constant Hodge--Tate--Sen weights.
\end{enumerate}

By \cite[Prop.\ 2.4]{Berg14} and the paragraphs before the proposition, the functors $D$ and $D^0$ are pro-representable, and in fact $\dim D^0(E[\varepsilon]/\varepsilon^2) = \dim D(E[\varepsilon]/\varepsilon^2)$.
Moreover, the tangent space of the eigencurve at the point $f$ injects into $D(E[\varepsilon]/\varepsilon^2)$, so we conclude that
\[
 \dim D^0(E[\varepsilon]/\varepsilon^2) = \dim D(E[\varepsilon]/\varepsilon^2)\geq 1.
\]
We will show in Corollary~\ref{cor:D0dim1} that in fact 
\[
 \dim_E D^0(E[\varepsilon]/\varepsilon^2) = \dim_E D(E[\varepsilon]/\varepsilon^2) = 1.
\]

Let $\widetilde{\rho}\colon \Gal_\QQ \rightarrow \GL(\widetilde{V})$ be in $D^0(E[\varepsilon]/\varepsilon^2)$.
Recall that we have
\[
 D(E[\varepsilon]/\varepsilon^2) \inj H^1(\Gal_\QQ, V\tensor V^\ast), \quad
\widetilde\rho=(\id+\rho'\varepsilon)\rho \mapsto \rho'. 
\]
We first reduce the Galois cohomology calculation from over $\QQ$ to over $K$.
\begin{lemma} \label{lem:restoK}
\[
 H^1(\QQ,V\tensor V^\ast) \isom H^1(K, V\tensor V^\ast)^{\Gal(K/\QQ)} \isom H^1(K,E) \directsum H^1(K, \psi(\psi^c)^{-1}).
\]
\end{lemma}
\begin{proof}
 We have the inflation-restriction sequence, 
\small
\[
  H^1(\Gal(K/\QQ), V\tensor V^\ast) \rightarrow H^1(\QQ,V\tensor V^\ast) \rightarrow H^1(K, V\tensor V^\ast)^{\Gal(K/\QQ)} \rightarrow H^2(\Gal(K/\QQ), V\tensor V^\ast).
\]
\normalsize
 Since $\Gal(K/\QQ)$ has order $2$, $H^1(\Gal(K/\QQ),V\tensor V^\ast)$ and $H^2(\Gal(K/\QQ),V\tensor V^\ast)$ are $2$-torsion.
 But they are also $E$-vector spaces, which are torsion-free. 
 Hence the first and last term vanish, and we get the first isomorphism in the lemma.
 
 For the second isomorphism in the lemma, note that the $\Gal(K/\QQ)$-action on $H^1(K,V\tensor V^\ast)$ is given by 
\[
 (c \rho')(g) = \rho(c) \cdot \rho'(c g c^{-1}) \cdot \rho(c)^{-1}. 
\]
 Because $\left.\rho\right|_{\Gal_K} \isom \psi \directsum \psi^c$, we have the isomorphism
\small
\[
 H^1(K,V\tensor V^\ast) \isom H^1(K,E) \directsum H^1(K,\psi(\psi^c)^{-1}) \directsum H^1(K,\psi^c\psi^{-1}) \directsum H^1(K,E),
\]
\normalsize
under which $\left.\rho'\right|_{\Gal_K}$ is mapped to $(e_{11},e_{12},e_{21},e_{22})$.
Then the invariance of $\left.\rho'\right|_{\Gal_K}$ under $\Gal(K/\QQ)$-action is translated into
\begin{align}
\label{eq:inv11}
 e_{22}(g)&=e_{11}(c g c^{-1}) \\
\label{eq:inv12}
 e_{21}(g)&=\eta(c)\eta'(c)^{-1} \cdot e_{12}(c g c^{-1}),
\end{align}
for all $g\in \Gal_K$.
Hence the projection onto the first two components 
\[
 H^1(\Gal_K, V\tensor V^\ast) \rightarrow H^1(\Gal_K,E)\directsum H^1(\Gal_K,\psi(\psi^c)^{-1})
\]
induces an isomorphism 
\[
H^1(\Gal_K,V\tensor V^\ast)^{\Gal(K/\QQ)}\xrightarrow{\sim} H^1(\Gal_K,E) \directsum H^1(\Gal_K,\psi(\psi^c)^{-1}). 
\]

\end{proof}

\begin{proposition} \label{prop:Ddim1}
 There is an isomorphism 
\[
 D(E[\varepsilon]/\varepsilon^2) \isom H^1_{\frakp}(\Gal_K,\psi(\psi^c)^{-1}) \isom H^1(\Gal_{K_{\frakp}},\psi(\psi^c)^{-1}).
\]
Here $H^1_\frakp(\Gal_K,\cdot)$ denotes the subspace of $H^1(\Gal_K,\cdot)$ of cohomology classes unramified/crystalline outside $\frakp$.

Hence $\dim_E D(E[\varepsilon]/\varepsilon^2)=\dim_E H^1(\Gal_{K_{\frakp}},\psi(\psi^c)^{-1})=1$.
\end{proposition}
\begin{proof}
First of all, from the minimal ramification condition in $D$, we know that all $e_{ij}$ are unramified outside $p$.

From the proof of \cite[Prop.\ 2.4]{Berg14}, we know that $e_{22}$ is crystalline at $\frakp$.
Hence Eq.~(\ref{eq:inv11}) relating $e_{11}$ and $e_{22}$ implies that $e_{11}$ is crystalline at $\frakp'$.

 On the other hand, by definition of $D$, $e_{11}\in H^1(K,E)$ regarded as a first deformation of the trivial character has constant $\frakp$-Hodge--Tate--Sen weight $0$.
 Sen's theorem (\cite[Cor.\ of Thm.\ 11]{Sen81}) implies that $e_{11}$ is potentially unramified at $\frakp$.
But the potentially unramified subspace of $H^1(\Gal_{\QQ_p},E)$ coincides with the crystalline subspace, so $e_{11}$ is crystalline at $\frakp$.

 In addition, $e_{12}\in H^1(\Gal_K, \psi (\psi^c)^{-1})$ has to be crystalline at $\frakp'$.
 This is because restricting to $\Gal_{K_{\frakp}}$, the Hodge--Tate weight of $\psi$ is $0$ and the Hodge--Tate weight of $\psi^c$ is $k-1$, so restricting to $\Gal_{K_{\frakp'}} = c\Gal_{K_{\frakp}}c^{-1}$, the Hodge--Tate weight $k-1$ of $\psi$ is greater than the Hodge--Tate weight $0$ of $\psi^c$.

Hence we have the following diagram.
\begin{center}
\begin{tikzcd}
 H^1(\Gal_\QQ, V\tensor V^\ast) \ar[r,"\sim"] & H^1(\Gal_K,E) \directsum H^1(\Gal_K,\psi (\psi^c)^{-1}) \\
 D(E[\varepsilon]/\varepsilon^2) \ar[u,hookrightarrow] \ar[r,hookrightarrow] & H^1_f(\Gal_K,E) \directsum H^1_{\frakp}(\Gal_K,\psi(\psi^c)^{-1}) \ar[u,hookrightarrow].
\end{tikzcd}
\end{center}
Here $H^1_f(\Gal_K,\cdot)$ (resp.\ $H^1_\frakp(\Gal_K,\cdot)$) denotes the subspace of $H^1(\Gal_K,\cdot)$ of cohomology classes unramified/crystalline everywhere (resp.\ unramified/crystalline outside $\frakp$).

It is known that $H^1_f(\Gal_K,E)=0$ because of the finiteness of the class group of $K$.
By the Iwasawa Main Conjecture for imaginary quadratic fields (\cite[Thm.\ 4.1]{Rub91}, \cite[Prop.\ III.1.3, III.1.4]{deS}), we also know that $H^1_f(\Gal_K,\psi(\psi^c)^{-1})=0$.
In fact, since $\psi(\psi^c)^{-1}$ is of infinite type $\sigma^{k-1}\bar\sigma^{-(k-1)}$ which is critical, the Main Conjecture and the interpolation property of $p$-adic $L$-functions (\cite[Thm.\ II.4.14, Cor.\  II.6.7]{deS}) says that $H^1_f(\Gal_K,\psi(\psi^c)^{-1})=0$ if and only if $L(\psi^{-1}\psi^c,0)\neq0$ and $\psi(\frakp)(\psi^c)^{-1}(\frakp)\neq p$, which is true.

 The vanishing of $H^1_f(\Gal_K,E)=0$ then implies
\[
 H^1_{\frakp}(\Gal_K,\psi(\psi^c)^{-1}) \stackrel{\loc_\frakp}{\rightarrow} H^1(\Gal_{K_\frakp},\psi(\psi^c)^{-1})
\]
is injective.

We conclude that the bottom map of the diagram becomes 
\[
 D(E[\varepsilon]/\varepsilon^2) \inj H^1(K_{\frakp},\psi(\psi^c)^{-1}).
\]

Since the eigencurve is equidimensional of dimension $1$, $\dim_E D(E[\varepsilon]/\varepsilon^2)\geq1$.
On the other hand, $\dim_E H^1(\Gal_{K_\frakp}, \psi(\psi^c)^{-1})=1$ by the local Euler characteristic formula since $\psi(\psi^c)^{-1}\neq E\text{ or } E(1)$,  
Hence $D(E[\varepsilon]/\varepsilon^2) \inj H^1(K_{\frakp},\psi(\psi^c)^{-1})$ must be an isomorphism.
\end{proof}

\begin{corollary} \label{cor:D0dim1}
 $\dim_E {D^0}(E[\varepsilon]/\varepsilon^2) = \dim_E D(E[\varepsilon]/\varepsilon^2) = 1$.
\end{corollary}

\subsection{Fourier coefficients}
 As in the introduction, let $f'=\sum_{n=1}^\infty a_n'q^n $ be the normalized generalized Hecke eigenform associated to $f$.
 Let $\widetilde{f} = f+f'\varepsilon$ and $\widetilde{\rho}\colon \Gal_\QQ\rightarrow \GL(\widetilde{V})$ the Galois representation associated to $\widetilde{f}$.
Here $\widetilde{V}$ is a free $E[\varepsilon]/\varepsilon^2$-module of rank $2$.
Then $\widetilde{\rho}$ is in ${D^0}(E[\varepsilon]/\varepsilon^2)$.

\begin{proposition} \label{prop:ap}
 $a_p'=0$.
\end{proposition}
\begin{proof}
 The $U_p$-eigenvalue of $\widetilde{f}$ is its $p$-th Fourier coefficient $a_p+a_p'\varepsilon$.
 In terms of Galois representation, the $U_p$-eigenvalue $a_p+a_p'\varepsilon$ is such that
\[
 D_\cris^+(\left.\widetilde{V}\right|_{\Gal_{K_\frakp}})^{\varphi=a_p+a_p'\varepsilon} \neq 0.
\]
Note that 
\[
 D_\cris^+(\left.\widetilde{V}\right|_{\Gal_{K_\frakp}})^{\varphi=a_p+a_p'\varepsilon} 
= D_\cris^+(\left.e_{22}\psi^c\right|_{\Gal_{K_\frakp}})^{\varphi=a_p+a_p'\varepsilon}.
\]
 But we have seen that $e_{22}\in H^1_f(\Gal_K,E) = 0$ in the proof of Proposition~\ref{prop:Ddim1}.
 As a result, $\left.e_{22}\psi^c\right|_{\Gal_{K_\frakp}} = \left.\psi^c\right|_{\Gal_{K_\frakp}}\tensor_E E[\varepsilon]/\varepsilon^2$, and hence $a_p'$ must be equal to $0$.
 \end{proof}

We have seen that there exists a basis $\{v_1, v_2 = \rho(c) v_1 \}$ of $V$ such that
\[
 \left.\rho\right|_{\Gal_K} \sim 
\begin{pmatrix}\psi&\\&\psi^c\end{pmatrix},
\quad
  \left.\rho\right|_{\Gal_\QQ \setminus \Gal_K} \sim 
\begin{pmatrix}&\eta'\\ \eta&\end{pmatrix},
\]
where $\eta$ and $\eta'$ are functions on $\Gal_\QQ\setminus \Gal_K$ given by $\eta(\cdot) = \psi(c\cdot)$ and $\eta'(\cdot) = \psi(\cdot c^{-1})$.

\begin{proposition} \label{prop:tr} \hfill
\begin{enumerate}
 \item For $g\in \Gal_K$,
$\Tr \widetilde{\rho}(g) = \psi(g)+\psi^c(g)$.
 \item For $g\in \Gal_\QQ \setminus \Gal_K$, 
$\Tr \widetilde{\rho}(g) = \eta(g)e_{12}(g^2)\varepsilon = \eta'(g)e_{21}(g^2)\varepsilon$.
\end{enumerate}
\end{proposition}
\begin{proof}
For (1), picking lifts $\widetilde{v_1}, \widetilde{v_2} \in\widetilde{V}$ of $v_1, v_2$, we may write 
\[
 \widetilde{\rho} \sim (\id+\begin{pmatrix}e_{11}&e_{12}\\e_{21}&e_{22}\end{pmatrix}\varepsilon)\begin{pmatrix}\psi&\\&\psi^c\end{pmatrix}.
\]
 We have seen that $e_{11},e_{22} \in H^1_f(\Gal_K,E) = 0$ in the proof of Proposition~\ref{prop:Ddim1}.
 Hence for any $g\in \Gal_K$, $\Tr \widetilde{\rho}(g) = \psi(g)+\psi^c(g)$.

 For (2), since $\rho'\in H^1(\Gal_\QQ, V\tensor V^\ast)$, we have
\[
 \rho'(g^2)=\rho'(g)+\rho(g)\rho'(g)\rho(g)^{-1}.
\]
Written in terms of matrices in the basis $\widetilde{v_1}$, $\widetilde{v_2}$, this says
\[
 \begin{pmatrix}
  e_{11}(g^2) & e_{12}(g^2) \\ e_{21}(g^2) & e_{22}(g^2)
 \end{pmatrix}
=
 \begin{pmatrix}
  e_{11}(g) & e_{12}(g) \\ e_{21}(g) & e_{22}(g)
 \end{pmatrix}
+ 
 \begin{pmatrix}
  e_{22}(g) & \frac{\eta'(g)}{\eta(g)}e_{21}(g) \\ \frac{\eta(g)}{\eta'(g)}e_{12}(g) & e_{11}(g)
 \end{pmatrix}.
\]
Hence 
\begin{align*}
 e_{11}(g^2)&=e_{11}(g)+e_{22}(g)\\
 e_{12}(g^2)&=e_{12}(g)+\frac{\eta'(g)}{\eta(g)}e_{21}(g). 
\end{align*}
In particular, $\Tr \widetilde{\rho}(g)=\eta(g)e_{12}(g)\varepsilon+\eta'(g)e_{21}(g)\varepsilon$ is equal to $\eta(g)e_{12}(g^2)\varepsilon=\eta'(g)e_{21}(g^2)\varepsilon$.
\end{proof}

\begin{corollary} \label{cor:al}
 \[
  a_\ell'= 
\begin{cases} 
0 & \text{if $\ell\nmid Np$ splits in $K$}\\ 
\psi(c\Frob_\ell)e(\Frob_\ell^2) &\text{if $\ell\nmid Np$ is inert in $K$.}
\end{cases}
 \]
Here $e\in H^1_{\frakp}(\Gal_K,\psi(\psi^c)^{-1})$ is the unique cohomology class mapping $\Frob_{\ell_0}^2\in \Gal_K$ to $\psi(c\Frob_{\ell_0})^{-1}$, where $\ell_0$ is the smallest rational prime which is inert in $K$.
\end{corollary}
\begin{proof}
 Note that whenever $\ell$ is a rational prime inert in $K$, $e_{12}(\Frob_\ell^2)$ is well-defined, independent of the choice of a cocycle representative of $e_{12} \in H^1(\Gal_K, \psi(\psi^c)^{-1})$.
This is because
\[
 \psi^c(\Frob_\ell^2)=\psi(c\Frob_\ell^2c^{-1})=\psi(\Frob_{c\ell c^{-1}}^2)=\psi(\Frob_\ell^2), 
\]

 Since $\Tr(\widetilde{\rho}(\Frob_\ell)) = a_\ell+a_\ell'\varepsilon$, the formula then follows from Proposition~\ref{prop:tr}.
 The choice of $e$ is because our normalization of $f'$ requires $a_{\ell_0}'=1$.
\end{proof}

\section{Numerical examples} \label{sec:num}
The following examples are generated by a code in the courtesy of Alan Lauder, written based on \cite{Lau}.
We acknowledge his help and the use of  Magma programming language.

\begin{example}
Let $K$ be the imaginary quadratic field $\QQ(\sqrt{-1})$.
Let $\psi$ be the Grossencharacter of $K$ with trivial conductor and of infinity type $(4,0)$.
Then $G_\psi$, the CM modular form associated to $\psi$, is of weight $k=5$, level $N=4$, and nebentypus $\chi\colon (\ZZ/4)^\times \rightarrow \CC^\times$ the non-trivial Dirichlet character of modulus $4$. 
The first few Fourier coefficients of $G_\psi$ are
\[
 G_\psi=q-4q^2+16q^4-14q^5-64q^8+81q^9+\cdots.
\]
We pick $p=5$, a prime which splits in $K$ and does not divide $N$.
Let $\alpha$, $\beta$ be the root of $X^2-a_p X+\chi(p)p^{k-1}=X^2+14X+5^4$ with $p$-adic valuation $k-1=4$ and $0$, respectively.
Let $f'$ be the (non-normalized) generalized Hecke eigenform with the same Hecke eigenvalues as $f(z)\defeq G_\psi(z)-\beta G_\psi(pz)$, scaled such that the leading coefficient $a_2'$ is $1$.
We computed numerically the first few non-zero Fourier coefficients $a_\ell'$ of $f'$ modulo $5^{24}$ for $\ell\nmid N$ a prime.

\begin{center}
\small
\begin{tabular}{r|r}
 \multicolumn{1}{c|}{$\ell$} & \multicolumn{1}{c}{$a_\ell' \bmod 5^{24}$} \\
\hline
  3 & 43300771101273669 \\
  7 & 43442244692236520 \\
 11 & 30279465837717252 \\ 
 19 & 11784730043200626 \\ 
 23 & 56240881617036337 \\ 
 31 & 18613606380354261 \\ 
 43 & 39991538540718615 \\ 
 47 & 53268861392126849 \\ 
 59 & 35400357120186448 \\ 
 67 & 31496794802809616 \\
 71 & 10538304364997549 \\ 
 79 & 19184781428210594 \\ 
 83 & 24773813366422376 
\end{tabular}
\end{center}
\end{example}

\begin{example}
Let $K$ be the imaginary quadratic field $\QQ(\sqrt{-3})$.
Let $\psi$ be the Grossencharacter of $K$ with trivial conductor and of infinity type $(6,0)$.
Then $G_\psi$, the CM modular form associated to $\psi$, is of weight $k=7$, level $N=3$, and nebentypus $\chi\colon (\ZZ/3)^\times \rightarrow \CC^\times$ the non-trivial Dirichlet character of modulus $3$. 
The first few Fourier coefficients of $G_\psi$ are
\[
 G_\psi=q-27q^3+64q^4-286q^7+729q^9+\cdots.
\]
We pick $p=7$, a prime which splits in $K$ and does not divide $N$.
Let $\alpha$, $\beta$ be the root of $X^2-a_p X+\chi(p)p^{k-1}=X^2+286X+7^6$ with $p$-adic valuation $k-1=6$ and $0$, respectively.
Let $f'$ be the (non-normalized) generalized Hecke eigenform with the same Hecke eigenvalues as $f(z)\defeq G_\psi(z)-\beta G_\psi(pz)$, scaled such that the leading coefficient $a_2'$ is $1$.
We compute numerically the first few non-zero Fourier coefficients $a_\ell'$ of $f'$ modulo $5^{22}$ for $\ell\nmid N$ a prime.
\begin{center}
\small
\begin{tabular}{r|r}
 \multicolumn{1}{c|}{$\ell$} & \multicolumn{1}{c}{$a_\ell' \bmod 7^{22}$} \\
\hline
 5 & 666108372229480561 \\
11 & 88592821880322831 \\
17 & 2092810930868948813 \\
23 & 1330989883549587564 \\
29 & 948498584988948579 \\
41 & 254724600121344265 \\
47 & 524234543371386261 \\
53 & 1745806937126778885 \\
59 & 3656628657475311802 \\
71 & 903737885018479401 \\
83 & 2252941180864123161 \\
89 & 2944581429297441793 \\
\end{tabular}
\end{center}
\end{example}

\section{Dimension of generalized Hecke eigenspace} \label{sec:ramdeg}
Let $f=\sum_{n\geq1} a_nq^n$ be as in the introduction, the critical p-stabilization of a CM modular form associated to a Grossencharacter $\psi$.
As before, $f$ is of weight $k$, level $\Gamma_1(N)\cap \Gamma_0(p)$, and nebentypus $\chi$.
Let $\rho_f \colon \Gal_\QQ \rightarrow \GL(V_f)$ be the $p$-adic Galois representation associated to $f$, where $V_f$ is a $2$-dimensional $E$-vector space for some finite extension $E$ of $\QQ_p$.

Recall that by \cite[Prop.\ 7.1]{Col96}, $f=\theta^{k-1}(g)$ for some $g\in M_{2-k}^\dagger(\Gamma_1(N))$.
Then $\rho_g\colon \Gal_\QQ \rightarrow \GL(V_g)$, the $p$-adic Galois representation associated to $g$, satisfies $\rho_g \isom \rho_f\tensor \chi_\cyc^{-k+1}$, where $\chi_\cyc$ is the $p$-adic cyclotomic character.

Let $e_f$ (resp.\ $e_g$) be the dimension of the generalized Hecke eigenspace in $M_k^\dagger(\Gamma_1(N))$ (resp.\ $M_{2-k}^\dagger(\Gamma_1(N))$) with Hecke eigenvalues the same as that of $f$ (resp.\ $g$).
The following lemma is a corollary of Coleman's  classicality theorem (\cite[Thm.\ 7.2]{Col96}).
Geometrically, $e_f$ (resp.\ $e_g$) is the ramification degree at the point $f$ (resp.\ $g$) of the weight map from the eigencurve to the weight space.

\begin{lemma} \label{lem:ramdeg}
 $e_f=e_g+1$.
\end{lemma}

We define a deformation functor $D^{\ord}$ on the category of Artinian local $E$-algebras with residue field $E$.
For any such $E$-algebra $A$, let $D^\ord(A)$ be the set of strict equivalence classes of representations $\rho_A\colon G_\QQ \rightarrow \GL(V_A)$ deforming $\rho_g$ such that
\begin{enumerate}
 \item for all primes $\ell$ not dividing $p$, $\rho_A$ and $\rho$ are the same after restricting to the inertia subgroup at $\ell$,
 \item $\left.V_A\right|_{G_{\QQ_p}}$ has a rank $1$ unramified quotient.
\end{enumerate}
Let $D^{\ord,0}$ be the sub-functor of $D^\ord$ of deformations with the additional condition
\begin{enumerate}[resume]
 \item
$\rho_A$ has constant Hodge--Tate--Sen weights.
\end{enumerate}

It is well-known that $D^{\ord}$ and $D^{\ord,0}$ are pro-representable (\cite[\S 30 Prop.\ 3]{Ma}).

We follow a similar strategy as in Section~\ref{subsec:Galoiscohom} to compute the Zariski tangent space of $D^{\ord,0}$.
Let $\widetilde{\rho_g}\colon G_\QQ \rightarrow \GL(\widetilde{V_g})$ be in $D^{\ord,0}(E[\varepsilon]/\varepsilon^2)$ and write $\widetilde{\rho_g}=(I_2+\rho'_g\varepsilon)\rho_g$.
Then $\rho'_g$ is an element of $H^1(\Gal_\QQ, V_g \tensor V_g^\ast)\isom H^1(\Gal_\QQ, V_f\tensor V_f^\ast)$, and corresponds to $(e_{11},e_{12},e_{21},e_{22})$ in 
\[
\left( H^1(K,E) \directsum H^1(K,\psi(\psi^c)^{-1}) \directsum H^1(K,\psi^c\psi^{-1}) \directsum H^1(K,E) \right)^{\Gal(K/\QQ)}
\]
as in the proof of  Lemma~\ref{lem:restoK}.

\begin{lemma} \label{lem:Dord}\hfill
\begin{enumerate} 
 \item $e_{11}$ is unramified outside $p$ and crystalline at both $\frakp$ and $\frakp'$.
 \item $e_{12}$ is unramified outside $\frakp$ and $\left.e_{12}\right|_{G_{K_{\frakp'}}}=0$.
\end{enumerate}
\end{lemma}
\begin{proof}
 First of all, from the minimal ramification condition in $D^{\ord,0}$, we know that all $e_{ij}$ are unramified outside $p$.
 Since $\psi^c\chi_\cyc^{-k+1}$ has $\frakp$-Hodge--Tate weight $0$, the ordinary condition in $D^{\ord,0}$ implies that $e_{22}$ is unramified (and hence crystalline) at $\frakp$ and that $\left.e_{21}\right|_{\Gal_{K_\frakp}}$ is $0$.
 Since $(e_{11},e_{12},e_{21},e_{22})$ is $\Gal(K/\QQ)$-fixed, $e_{22}$ being crystalline at $\frakp$ is equivalent to $e_{11}$ being crystalline at $c\frakp c^{-1}=\frakp'$, and $\left.e_{21}\right|_{\Gal_{K_\frakp}}=0$ is equivalent to $\left.e_{12}\right|_{\Gal_{K_{\frakp'}}}=0$.
 On the other hand, from the condition of fixed Hodge--Tate--Sen weights in $D^{\ord,0}$ we know that $e_{11}$ is crystalline at $\frakp$, by the same argument as the third paragraph of Proposition~\ref{prop:Ddim1}.
 Also $e_{12}\in H^1(K, \psi(\psi^c)^{-1})$ has to be crystalline at $\frakp'$ because the $\frakp'$-Hodge--Tate weight $k-1$ of $\psi$ is greater than the $\frakp'$-Hodge--Tate weight $0$ of $\psi^c$.

\end{proof}

\begin{proposition} \label{prop:ramdeg}
Let $\bar{\rho}_f$ be the reduction of $\rho_f$ modulo $p$.
Assume that $\bar\rho_f$ is absolutely irreducible after restricted to $\Gal_{\QQ(\mu_p)}$.
Then the following are equivalent:
\begin{enumerate}
 \item \label{cond:e_f} $e_f=2$;
 \item \label{cond:e_g} $e_g=1$;
 \item \label{cond:cohom} The restriction to decomposition group at $\frakp$
\[
H^1_{\frakp}(K,\psi(\psi^c)^{-1}) \rightarrow H^1(K_{\frakp'},\psi(\psi^c)^{-1})
\] 
is an isomorphism (equivalently, non-zero).
\end{enumerate}
\end{proposition}
\begin{proof}
 The equivalence of (\ref{cond:e_f}) and (\ref{cond:e_g}) follows from Lemma~\ref{lem:ramdeg}.

The tangent space of the eigencurve at the point $g$ is isomorphic to $D^{\ord}(E[\varepsilon]/\varepsilon^2)$ by the big $R=\TT$ theorem (\cite[Thm.\ 5.29]{Hida_book}).
Hence $e_g=1$ is equivalent to $\dim_E D^{\ord,0}(E[\varepsilon]/\varepsilon^2)=0$.

 By Lemma~\ref{lem:restoK} and Lemma~\ref{lem:Dord}, we have
\[
 D^{\ord,0}(E[\varepsilon]/\varepsilon^2) \inj H^1_f(K,E) \directsum \ker\left[H^1_{\frakp}(K,\psi(\psi^c)^{-1})\rightarrow H^1(K_{\frakp'},\psi(\psi^c)^{-1})\right].
\]
 It is known that $H^1_f(K,E)$ is zero because of the finiteness of the class group of $K$.
 From the Iwasawa Main Conjecture for imaginary quadratic fields (\cite[Thm.\ 4.1]{Rub91}, \cite[Prop.\ III.1.3, III.1.4]{deS}), we have $H^1_f(K,\psi(\psi^c)^{-1}) = 0$, and hence 
\[
 H^1_{\frakp}(K,\psi(\psi^c)^{-1}) \rightarrow H^1(K_{\frakp},\psi(\psi^c)^{-1})
\]
is injective, the latter being $1$-dimensional.
 In fact, we have seen in Prop.~\ref{prop:Ddim1} that this is an isomorphism, i.e., $\dim_E H^1_{\frakp}(K,\psi(\psi^c)^{-1})=1$.
 
Hence $\dim_E D^{\ord,0}(E[\varepsilon]/\varepsilon)=0$ if and only if 
\[
 \ker\left[H^1_{\frakp}(K,\psi(\psi^c)^{-1})\rightarrow H^1(K_{\frakp'},\psi(\psi^c)^{-1})\right]=0,
\]
and if and only if $H^1_{\frakp}(K,\psi(\psi^c)^{-1}) \rightarrow H^1(K_{\frakp'},\psi(\psi^c)^{-1})$ is an isomorphism.
Since both the source and target are $1$-dimensional, this is equivalent to saying that the map $H^1_{\frakp}(K,\psi(\psi^c)^{-1}) \rightarrow H^1(K_{\frakp'},\psi(\psi^c)^{-1})$ is non-zero.
\end{proof}

\begin{remark}
 In \cite[Prop.\ 1]{Bel}, Bella\"iche used the method of reducibility ideal to conclude the same result.
\end{remark}

\section*{Acknowledgments}
The author would like to thank her advisor Barry Mazur for his constant support. 
She particularly wants to thank Alan Lauder for referring her to their paper \cite{DLR15} and providing many support on Magma computation.
The author was partially supported by the Government Scholarship to Study Abroad from Taiwan.


\begin{thebibliography}{10}

\bibitem{BD}
Adel Betina and Mladen Dimitrov.
\newblock {Geometry of the eigencurve at CM points and trivial zeros of Katz $p$-adic $L$-functions}.
\newblock {\em 	arXiv:1907.09422}, Sep 2019.

\bibitem{Bel}
Jo\"{e}l Bella\"{i}che.
\newblock Computation of the critical $p$-adic $L$-functions of CM modular
  forms.
\newblock Preprint, available at
  \url{http://people.brandeis.edu/~jbellaic/preprint/CML-functions4.pdf}.

\bibitem{Berg14}
John Bergdall.
\newblock Ordinary modular forms and companion points on the eigencurve.
\newblock {\em J. Number Theory}, 134:226--239, 2014.

\bibitem{Bet16}
Adel {Betina}.
\newblock {Les Vari\'{e}t\'{e}s de Hecke-Hilbert aux points classiques de poids
  $1$}.
\newblock {\em J. Th\'{e}or. Nombres Bordeaux}, 30(2):575--607, 2018.

\bibitem{CV}
S.~Cho and V.~Vatsal.
\newblock Deformations of induced {G}alois representations.
\newblock {\em J. Reine Angew. Math.}, 556:79--98, 2003.

\bibitem{Col96}
Robert~F. Coleman.
\newblock Classical and overconvergent modular forms.
\newblock {\em Invent. Math.}, 124(1-3):215--241, 1996.

\bibitem{DLR15}
Henri Darmon, Alan Lauder, and Victor Rotger.
\newblock Overconvergent generalised eigenforms of weight one and class fields
  of real quadratic fields.
\newblock {\em Adv. Math.}, 283:130--142, 2015.

\bibitem{deS}
Ehud de~Shalit.
\newblock {\em Iwasawa theory of elliptic curves with complex multiplication},
  volume~3 of {\em Perspectives in Mathematics}.
\newblock Academic Press, Inc., Boston, MA, 1987.
\newblock $p$-adic $L$ functions.

\bibitem{Hida_book}
Haruzo Hida.
\newblock {\em Modular forms and {G}alois cohomology}, volume~69 of {\em
  Cambridge Studies in Advanced Mathematics}.
\newblock Cambridge University Press, Cambridge, 2000.

\bibitem{Lau}
Alan G.~B. Lauder.
\newblock Computations with classical and {$p$}-adic modular forms.
\newblock {\em LMS J. Comput. Math.}, 14:214--231, 2011.

\bibitem{Ma}
Barry Mazur.
\newblock An introduction to the deformation theory of {G}alois
  representations.
\newblock In {\em Modular forms and {F}ermat's last theorem ({B}oston, {MA},
  1995)}, pages 243--311. Springer, New York, 1997.

\bibitem{Rub91}
Karl Rubin.
\newblock The ``main conjectures'' of {I}wasawa theory for imaginary quadratic
  fields.
\newblock {\em Invent. Math.}, 103(1):25--68, 1991.

\bibitem{Sen81}
Shankar Sen.
\newblock Continuous cohomology and {$p$}-adic {G}alois representations.
\newblock {\em Invent. Math.}, 62(1):89--116, 1980/81.

\end{thebibliography}
\end{document}